\documentclass[12pt,a4paper]{article}

\usepackage{amssymb}
\usepackage{amsmath, amsthm}
\usepackage{arydshln}
\usepackage{epsfig}
\usepackage{setspace}
\usepackage{comment}
\usepackage{geometry}
\usepackage{hyperref}
\usepackage{mathtools}
\usepackage{mleftright}\mleftright

\def\RR{\mathbb{R}}

\def\ZZ{\mathbb{Z}}
\def\QQ{\mathbb{Q}}
\def\FF{\mathbb{F}}
\def\KK{\mathbb{K}}
\def\11{\mathbf{1}}

\numberwithin{equation}{section}

\newcommand{\rank}{\mathop{\rm rank} }

\newcommand{\supp}{\mathop{\rm supp} }

\newcommand{\Conv}{\mathop{\rm Conv} }

\DeclareMathOperator{\Det}{Det}
\DeclareMathOperator{\rk}{rk}
\DeclareMathOperator{\ncrk}{ncrk}
\newcommand{\ncrank}{\mathop{\rm nc\mbox{-}rank} }

\newcommand{\proj}{{\rm proj}}

\newcommand{\ncdet}{\mathop{\rm Det} }

\newtheorem{Thm}{Theorem}[section]

\newtheorem{Lem}[Thm]{Lemma}

\theoremstyle{definition}
\newtheorem*{Clm}{Claim}

\newtheorem{Rem}[Thm]{Remark}

\title{On integral polytopes related to Edmonds' problem}
\author{Hiroshi Hirai\footnote{
Graduate School of Mathematics,
Nagoya University, Nagoya, 464-8602, Japan.
\texttt{\footnotesize hirai.hiroshi@math.nagoya-u.ac.jp}
}}

\begin{document}

\maketitle

\begin{abstract}
    In this paper, we study polyhedral aspects on commutative and noncommutative Edmonds' problems for computing the rank of linear symbolic matrix $A = \sum_{k=1}^m A_k x_k$. We regard them as linear optimization over integral polytopes ${\cal P}(A)$ and ${\cal Q}(A)$, respectively, which are obtained by the convex hulls of exponent vectors of subdeterminants of~$A$ and 
    its blow-ups $A^{\{d\}} = \sum_{k=1}^m A_k \otimes X_k$ $(d=1,2,\ldots)$.
    By extending previously known results on nc-rank, we establish a hierarchy of integral 
    polytopes ${\cal P}(A) \subseteq {\cal P}^{\leq 2}(A) \subseteq {\cal P}^{\leq 3}(A) \subseteq \cdots = {\cal Q}(A)$ and show that the integrality gap of ${\cal Q}(A)$ relative to ${\cal P}(A)$ is at least $1/2$.
    Further, 
    we show that if each $A_k$ is rank-2 skew-symmetric, then the above hierarchy terminates at the second level and the integrality gap is improved to $2/3$.
\end{abstract}

\noindent
Keywords: Edmonds' problem, noncommutaive rank, Dieudonn\'e determinant, free skew field, generic division algebra, fractional linear matroid matching\\

\noindent
MSC-classifications: 52B12, 12E15

\section{Introduction}

{\em Edmonds' problem}~\cite{Edmonds67}
asks to compute the rank of a linear symbolic matrix
\begin{equation}\label{eqn:A}
	A = \sum_{k=1}^m A_k x_k,
\end{equation}
where  $A_k$ are $n \times n$ matrices over a field $\KK$, 
$x_k$ are indeterminates, and the rank $\rk A$ is considered in the rational function field $\KK(x_1,x_2,\ldots,x_m)$.
 This problem plays an important role in combinatorics and computer science; see e.g., \cite{Kabanets2004,Lovasz89,MurotaMatrix}.
Edmonds' problem 
is not known to be polynomial-time solvable.
Recently, a noncommutative version of Edmonds' problem, in which $x_ix_j \neq x_j x_i$,
was introduced by~\cite{IQS15a}, and the corresponding rank, 
{\em noncommutative rank (nc-rank)}, was shown to be computable in polynomial time~\cite{GGOW15,HamadaHirai,IQS15b}.
Since then, there have been a series of interesting developments 
in the literature;
see e.g., \cite{ArvindChatterjeeMukhopadhyay2024FOCS,BFGOWW} and reference therein.

In this paper, we address ``polyhedral" aspects on commutative and noncommutative Edmonds' problems. This is inspired by polyhedral methods in combinatorial optimization~\cite{SchrijverBook}. 
Let us formulate Edmonds' problem as 
linear optimization on a polytope:
For a multivariate polynomial $p(x_1,x_2,\ldots,x_m)= \sum_{u \in \ZZ_{\geq 0}^m} a_{u}x_1^{u_1}x_2^{u_2} \cdots x^{u_m}_m$,
let $\supp p \subseteq \ZZ_{\geq 0}^m$ denote the set of
all exponent vectors $u$ with  $a_{u} \neq 0$.
For a matrix~$A$ in~(\ref{eqn:A}),
define polytope ${\cal P}(A)$ as the convex hull of all exponent vectors of all subdeterminants of $A$:
\begin{equation*}
{\cal P}(A) := \Conv \bigcup_{I,J \subseteq [n]: |I|=|J|} \supp \det A[I,J],  
\end{equation*}
where $A[I,J]$ denotes the submatrix of $A$ specified by row index set $I$ and column index set $J$, and $\det A[\emptyset, \emptyset] :=1$.
It is clear that $\rk A$ is obtained by linear optimization on ${\cal P}(A)$ with the all-one objective vector ${\bf 1}$:
\begin{equation*}
\max \{ {\bf 1}^{\top}u \mid u \in {\cal P}(A) \} = \rk A.
\end{equation*}
To obtain a better understanding of ${\cal P}(A)$, 
we consider weighted maximization on it, which corresponds to
computing the maximum degree of subdeterminants.
For an integer weight $c \in \ZZ^m$, define $A[c]$ by
\begin{equation}\label{eqn:A[c]}
A[c] := \sum_{k=1}^m A_k x_kt^{c_k},
\end{equation}
where $t$ is a new indeterminate. 
Then it holds
\begin{equation}\label{eqn:weighted_max}
\max \{ c^{\top}u \mid u \in {\cal P}(A) \} = 
\max_{I,J \subseteq [n]:|I|=|J|} \deg \det (A[c])[I,J],
\end{equation}
where $\deg$ means the maximum degree in $t$.
In this regard, computing the maximum degree of subdeterminants 
can be viewed as a weighted extension of Edmonds' problem.

A series of papers~\cite{HH_degdet,HiraiIkeda,HIOS2025}
has developed a noncommutative version of weighted Edmonds' problem, 
and established an analogue of 
the relation (\ref{eqn:weighted_max}) in this setting, where another polytope ${\cal Q}(A)$ plays the role of ${\cal P}(A)$.
This polytope ${\cal Q}(A)$ is a relaxation of ${\cal P}(A)$:
\begin{equation*}
{\cal P}(A) \subseteq {\cal Q}(A).
\end{equation*}
From the viewpoint of combinatorial optimization,
it is a ``tractable" relaxation:
Linear optimization on ${\cal Q}(A)$ admits a nice duality (good characterization) and is polynomial-time solvable.
The paper \cite{HIOS2025} examined this viewpoint for several combinatorial optimization problems
that can be formulated as instances of Edmonds' problem. 

The present paper continues to study the two polytopes ${\cal P}(A), {\cal Q}(A)$, and shows four results extending previously known for nc-rank.
Here ${\cal Q}(A)$ is defined as the union of the projections 
of ${\cal P}(A^{\{d\}})$ for the {\em $d$-blow-up} $A^{\{d\}} = \sum_{k=1}^m A_k \otimes X_k$ of $A$ $(d=1,2,\ldots)$, where $X_k$ is 
a $d \times d$ matrix of indeterminate entries.
We show that the polytopes obtained as partial union are also integral polytopes (Theorem~\ref{thm:integrality}). This can be viewed as an extension of the {\em regularity lemma}~\cite{IQS15a}, and yields 
a hierarchy of integral polytopes 
${\cal P}(A) \subseteq {\cal P}^{\leq 2}(A) \subseteq {\cal P}^{\leq 3}(A) \subseteq \cdots = {\cal Q}(A)$. We show that the integrality gap of ${\cal Q}(A)$ relative to ${\cal P}(A)$ is at least $1/2$ (Theorem~\ref{thm:2-approx_poly}).
This is an extension of the result~\cite{FortinReutenauer04} that nc-rank is $1/2$-approximation of rank.
We then focus on the case where each $A_k$ is rank-2 skew-symmetric. 
In this case, ${\cal P}(A)$ and ${\cal Q}(A)$ 
correspond to linear matroid matching and 
its fractional relaxation, respectively; see \cite{HIOS2025,Lovasz89,OkiSoma_SICOMP}.  
We show that the above hierarchy terminates at the second level (Theorem~\ref{thm:2nd_blowup}).
This is an extension of~\cite{OkiSoma_SICOMP} that nc-rank ($=$ the maximum size of fractional matroid matchings) is obtained by the second blow-up.
We finally show that the integrality gap of ${\cal Q}(A)$ is improved to $2/3$ (Theorem~\ref{thm:3/2-approx}).
This is an extension of the well-known fact
that a natural LP relaxation of matchings in undirected graphs has integrality gap $2/3$ but may be unexpected from the fact~\cite{KaparisLetchfordMourtos2020,LeeSviridenkoVondrak2013}
that fractional matroid matching relaxation 
has unbounded integrality gap in general.

In Section~\ref{sec:pre}, 
we provide necessary background
on integral polytopes, noncommutaive rank, 
and a noncommutative version of the degree of determinants.
In Section~\ref{sec:results}, we present the results and proofs.

\section{Preliminaries}\label{sec:pre}
%
\subsection{Integral polytopes}
We will utilize some elementary properties of (integral) 
polytopes.
For a polytope (bounded polyhedron) ${\cal P} \subseteq \RR^m$, 
the support function $f_{\cal P}: \RR^m \to \RR$
is defined by $f_{\cal P}(c) := \max \{c^{\top}u \mid u \in {\cal P}\}$ for $c \in \RR^m$.
A polytope ${\cal P}$ is said to be {\em rational}/{\em integral} if every extreme point is a rational/integer vector.
\begin{Lem}[{Edmonds-Giles; see \cite[Corollary 22.1a]{SchrijverIP}}]\label{lem:EdmondsGiles}
A rational polytope ${\cal P}$ is integral if and only 
if $f_{\cal P}(c)$ is an integer for every integer vector $c \in \ZZ^m$.
\end{Lem}
We will study inclusion of two polytopes by means 
of their support functions.
It is basic in convex analysis that
polytope ${\cal P}$ is recovered from 
$f_{\cal P}$ by ${\cal P} = \{u \in \RR^m \mid c^{\top}u \leq f_{\cal P}(c) \ (c \in \RR^m)\}$, and that $f_{\cal P}$ is continuous convex 
and positively homogeneous 
$\alpha f_{\cal P}(c) = f_{\cal P}(\alpha c) =  f_{\alpha{\cal P}}(c)$ for $\alpha \in \RR_{\geq 0}$, where $\alpha {\cal P} := \{ \alpha u \mid u \in {\cal P}\}$.
Thus we have:
\begin{Lem}\label{lem:P_subseteq_Q}
For polytopes ${\cal P}, {\cal Q} \subseteq \RR^m$ and 
$\alpha \in \RR_{\geq 0}$, it holds
${\cal Q} \subseteq \alpha {\cal P}$ if and only if $f_{\cal Q}(c) \leq \alpha f_{\cal P}(c)$ for every integer vector $c \in \ZZ^m$.
\end{Lem}

\subsection{Noncommutative rank}
Here we recall nc-rank and its properties which we address. 
Regard the matrix $A$ in (\ref{eqn:A}) 
as a matrix over the noncommutative polynomial ring 
$\KK\langle x_1,\ldots,x_m \rangle$
over $\KK$.
This ring 
admits the most generic skew field $\KK(\langle x_1,\ldots,x_m \rangle)$ of fraction, 
called the {\em free skew field}; see~\cite{Cohn}.
{\em Noncommutative Edmonds' problem} by Ivanyos, Qiao, and Subrahmanyam~\cite{IQS15a} is the problem of computing the rank of $A$ considered in the free skew field, 
where the corresponding rank of $A$ is called the {\em noncommutative rank (nc-rank)} and is denoted by $\ncrk A$.
It is known~\cite{Cohn} that $\ncrk A$ equals the minimum $r$ with
$A = BC^{\top}$ for $B,C \in \KK \langle x_1,\ldots,x_m \rangle^{n \times r}$.
This is sharpened by Fortin and Reutenauer~\cite{FortinReutenauer04} as follows.
\begin{Thm}[\cite{FortinReutenauer04}]\label{thm:FortinReutenauer}
\begin{eqnarray}
\ncrk A & = & \min \{2n-|I|-|J|  \mid \nonumber \\ 
&& \quad \  (SAT)[I,J] = O,\nonumber \\
&& \quad \ S,T \in GL_n(\KK),\ I,J \subseteq [n]\}. \label{eqn:formula1}
\end{eqnarray}
\end{Thm}
Then, one can see from (\ref{eqn:formula1}) (or (\ref{eqn:formula2}) below) that
\begin{equation*}
\rk A \leq \ncrk A.
\end{equation*}
Further application of (\ref{eqn:formula1}) shows that 
nc-rank is 1/2-approximation of rank:
\begin{Thm}[\cite{FortinReutenauer04}]\label{thm:2-approx_nc}
\[
\frac{1}{2} \ncrk A \leq  \rk A.
\]
\end{Thm}
The free skew field is realized as the skew field of {\em noncommutative rational functions}~\cite{Kaliuzhnyi-VerbovetskyiVinnikov2012}, which
are formal {\em rational expressions}, 
such as $(x_1x_2^{-1}x_1 + a x_3)^{-1} + x_2$, modulo the equivalence under substitution $x_k \leftarrow Z_k$ of all square matrices $Z_k \in \KK^{d \times d}$ of all sizes~$d$.
Motivated by matrix substitution, 
the $d$-th {\em blow-up} $A^{\{d\}}$ of $A$ is defined by
\begin{equation*}
A^{\{d\}} := \sum_{k=1}^m A_k \otimes X_k, 
\end{equation*}
where $X_k = (x_{k,ij})$ is 
a $d \times d$ matrix
of indeterminate entries $x_{k,ij}$ and 
$\otimes$ means the Kronecker product.
The rank of $A^{\{d\}}$ is considered on 
the rational function field $\KK(\{x_{k,ij}\})$ with indeterminates $x_{k,ij}$ $(k \in [m],i,j \in [d])$.
\begin{Thm}[{\cite{IQS15a}}]\label{thm:blow-up}
\begin{equation}\label{eqn:formula2}
\ncrk A = \max_{d=1,2,\ldots}\frac{1}{d} \rk A^{\{d\}}.
\end{equation}
\end{Thm}
Moreover, $(1/d) \rk A^{\{d\}}$ is always an integer. This property, called the {\em regularity lemma}, plays important roles in nc-rank theory.
\begin{Thm}[{Regularity Lemma~\cite{IQS15a}}]\label{thm:regularity}
$\rk A^{\{d\}} \in d \ZZ$.
\end{Thm}
Derksen and Makum~\cite{DerksenMakam2018} gave an elegant proof of the regularity lemma by use of the {\em generic division algebra}. 
They also showed that the regularity lemma yields a linear upper bound in (\ref{eqn:formula2}). 
\begin{Thm}[\cite{DerksenMakam2017}]\label{thm:d<=n-1}
The maximum in (\ref{eqn:formula2}) is attained by $d \leq n-1$.
\end{Thm} 
A quick proof for slightly weaker bound $d \leq n$, via the regularity lemma, is given by~\cite[Lemma 5.2]{IQS15b}.

\subsection{Degree of the Dieudonn\'e determinant}

Next we recall a noncommutative version of weighted Edmonds' problem 
developed in~\cite{HH_degdet,HiraiIkeda,HIOS2025}; 
see also \cite{Oki}.
Generalizing $A[c]$, 
we consider the following matrix 
\begin{equation}\label{eqn:B}
B = \sum_{k=1}^m B_kx_k,
\end{equation}
where $B_k = B_k(t) \in \KK(t)^{n \times n}$ 
are $n \times n$ matrices
over the rational function field $\KK(t)$.
We regard $B$ as a matrix over the rational function skew field $(\KK(\langle x_1,\ldots,x_m\rangle))(t)$, which is the fraction of the skew polynomial ring over $\KK(\langle x_1,\ldots,x_m\rangle)$ with commutative indeterminate $t$.
For matrices over a skew field $\FF$,
we consider the {\em Dieudonn\'e determinant}~\cite[Section 9.2]{Cohn_Algebra},
which takes a value in $\FF$ modulo the commutator subgroup of the multiplicative group $\FF \setminus \{0\}$. 
For our case of $\FF= (\KK(\langle x_1,\ldots,x_m\rangle))(t)$, 
the maximum degree $\deg$ with respect to $t$ of the Dieudonn\'e determinant is well-defined.
We use $\Det$ to denote the Dieudonn\'e determinant.
The paper~\cite{HH_degdet} established 
a formula for $\deg \Det B$ analogous to (\ref{eqn:formula1}). 
We will use its variant for 
the maximum degree of subdeterminants in \cite{HIOS2025}; see Theorem~\ref{thm:degDet_B} below.
As before, 
the $d$-th blow-up $B^{\{d\}}$ of $B$ is defined by
\begin{equation*}
B^{\{d\}} = \sum_{k=1}^m B_k \otimes X_k,
\end{equation*}
where $X_k = (x_{k,ij})$ are $d \times d$ indeterminate matrices.
Then, an analogue of Theorems~\ref{thm:blow-up} and \ref{thm:d<=n-1}
is the following: 
\begin{Thm}[\cite{HiraiIkeda}]\label{thm:degDet_B_blowup}
\begin{equation}\label{eqn:degDet_formula1}
    \deg \ncdet B =  \max_{d =1,2,\ldots}\frac{1}{d}\deg \det B^{\{d\}}, 
\end{equation}
where the maximum is attained by $d \leq n-1$.
\end{Thm}
Define the maximum degrees $\delta_{\rm max}(B)$ and $\varDelta_{\rm max}(B)$ of subdeterminants of $B$ by 
\begin{eqnarray}
\delta_{\rm max}(B) &:= &\max_{I,J \subseteq [n]:|I|=|J|} \deg \det B[I,J], \nonumber \\
\varDelta_{\rm max}(B) &:= &\max_{I,J \subseteq [n]: |I|=|J|} \deg \Det B[I,J] \nonumber \\
&=& \max_{d =1,2,\ldots} \max_{I,J \subseteq [n]: |I|=|J|} 
\frac{1}{d}\deg \det B[I,J]^{\{d\}}.
\label{eqn:Delta(B)} 
\end{eqnarray}
We describe a dual formula of $\varDelta_{\rm max}(B)$ as follows.
For $p \in (\ZZ \cup \{-\infty\})^n$, let $(t^{p})$ 
denote the diagonal matrix with diagonals $t^{p_1},t^{p_2},\ldots,t^{p_n}$ in order, where we let $t^{-\infty} := 0$.
Let $\KK(t)^{-} \subseteq \KK(t)$ denote the ring of 
rational functions $p$ with $\deg p \leq 0$, 
and let $GL_n(\KK(t)^{-})$ denote the group 
of matrices $M \in (\KK(t)^-)^{n \times n}$ with $M^{-1} \in (\KK(t)^-)^{n \times n}$ 
(that is equivalent to $\deg \det M = 0$).

\begin{Thm}[\cite{HIOS2025}]\label{thm:degDet_B}
\begin{eqnarray}
\varDelta_{\rm max}(B) &=& \min \{  {\bf 1}^{\top} p + {\bf 1}^{\top}q \mid \nonumber \\
&& \quad \deg ((t^{-p}) PB_k Q(t^{-q}))_{ij}\leq  0 \ (i,j \in [n],k\in [m]), \nonumber \\
&& \quad p,q \in \ZZ_{\geq 0}^n,\  P,Q \in GL_n(\KK(t)^-) \}. \label{eqn:degDet_formula2}
\end{eqnarray}
\end{Thm}
When $B= A[c]$, 
two formulas (\ref{eqn:Delta(B)}) and 
(\ref{eqn:degDet_formula2}) are interpreted as a duality for linear optimization on the polytope ${\cal Q}(A)$ determined from $\supp \det A[I,J]^{\{d\}}$.
Notice that 
$\det A[I,J]^{\{d\}}$ is a polynomial of variables
$x_{k,ij}$ for $k \in [m], i,j \in [d]$.
If we write the exponent of $x_{k,ij}$ by $z_{k,ij}$, then
$\supp \det A[I,J]^{\{d\}}$ consists of $md^2$-dimensional integer vectors $z = (z_{k,ij})_{k \in [m],i,j \in [d]}$.
For such a vector $z= (z_{k,ij})_{k \in[m],i,j \in [d]}$, define
the projection $\proj_d(z)\in \QQ^m_{\geq 0}$ by
\begin{equation*}
	\proj_d (z)_k \coloneqq   \frac{1}{d} \sum_{i,j \in [d]} z_{k,ij} \quad (k \in [m]).
\end{equation*}
Then it holds
\begin{equation*}
\frac{1}{d}\deg \det (A[c])[I,J]^{\{d\}} = \max \{c^{\top} \proj_d(z) \mid z \in \supp \det A[I,J]^{\{d\}} \}.
\end{equation*}
For $d =1,2,\ldots,$ and $\ell = 0,1,\ldots,n$, define  subset ${\cal S}_\ell^{d} \subseteq \QQ_{\geq 0}^m$ by
\begin{equation*}
{\cal S}_\ell^{d} := \bigcup_{I,J \subseteq [n]: |I|=|J|=\ell} \proj_d(\supp \det A[I,J]^{\{d\}}).
\end{equation*}
Define ${\cal Q}(A)$ by the convex hull of the union of all ${\cal S}_{\ell}^d$:
\begin{equation}\label{eqn:Q(A)}
{\cal Q}(A) \coloneqq  \Conv \bigcup_{d=1}^\infty \bigcup_{\ell=0}^n  {\cal S}_\ell^d.
\end{equation}
It is clear that ${\cal P}(A) \subseteq {\cal Q}(A)$.
By (\ref{eqn:degDet_formula1}), it holds 
\begin{equation}
\varDelta_{\rm max}(A[c]) = \max \{ c^{\top}u \mid u \in {\cal Q}(A)\} \  (= f_{{\cal Q}(A)}(c)).
\end{equation}
Theorem~\ref{thm:degDet_B} is sharpened as follows:
\begin{Thm}[\cite{HIOS2025}]
\begin{eqnarray}
\varDelta_{\rm max}(A[c]) &= &  \min \{  {\bf 1}^{\top} p + {\bf 1}^{\top}q \mid \nonumber \\
&& \quad \deg ((t^{-p}) SA_k t^{c_k}T(t^{-q}))_{ij}\leq  0 \ (i,j \in [n],k\in [m]), \nonumber \\
&& \quad p,q \in \ZZ_{\geq 0}^n,\  S,T \in GL_n(\KK) \}. \label{eqn:A[c]_dual}
\end{eqnarray}
\end{Thm}
By Theorem~\ref{thm:degDet_B_blowup}, the range of $d$ 
in (\ref{eqn:Q(A)}) may be replaced by $1\leq d \leq n-1$, 
and hence ${\cal Q}(A)$ is a rational polytope. 
By (\ref{eqn:A[c]_dual}), $f_{{\cal Q}(A)}(c) (= \varDelta_{\rm max}(A[c]))$ is an integer for every $c \in \ZZ^n$. By Lemma~\ref{lem:EdmondsGiles}, we have:
\begin{Thm}[\cite{HIOS2025}]
    ${\cal Q}(A)$ is an integral polytope.
\end{Thm}

\section{Results}\label{sec:results}

\subsection{Polyhedral regularity lemma}
We define intermediate polytopes
${\cal P}^d_\ell (A)$, ${\cal P}^d(A)$, and ${\cal P}^{\leq d}(A)$ toward ${\cal Q}(A)$:
\begin{equation}
{\cal P}^d_\ell(A) \coloneqq   \Conv {\cal S}_\ell^d, \quad {\cal P}^{d}(A) \coloneqq  \Conv \bigcup_{\ell=0}^n {\cal S}_\ell^d, \quad
 {\cal P}^{\leq d}(A) \coloneqq  \Conv \bigcup_{f=1}^d \bigcup_{\ell=0}^n {\cal S}_\ell^f,
\end{equation}
We show that they are also integral polytopes. 
This can be viewed as an extension of the regularity lemma (Theorem~\ref{thm:regularity}). 
\begin{Thm}\label{thm:integrality}
    ${\cal P}_\ell^d(A)$, ${\cal P}^{d}(A)$, and ${\cal P}^{\leq d}(A)$ are all integral polytopes.
\end{Thm}
Thus, we obtain the following hierarchy of
integral polytopes 
containing ${\cal P}(A)$:
\begin{equation}
{\cal P}(A) = {\cal P}^{\leq 1}(A) \subseteq {\cal P}^{\leq 2}(A) 
\subseteq \cdots \subseteq {\cal P}^{\leq n-1}(A)  = \cdots = {\cal Q}(A).
\end{equation}
We expect that the minimum $r$ with ${\cal P}^{\leq r}(A) = {\cal Q}(A)$
reflects the complexity of Edmonds' problem for $A$.
One example of $r=1$ is the case where each $A_k$ is rank one.
In this case, ${\cal P}(A)= {\cal Q}(A)$ is the polytope of linear matroid intersection~\cite[Section 3.2.4]{HIOS2025}.
We will see an example of $r=2$ in Section~\ref{subsec:rank2-skew-symmetric}.
\paragraph{Proof of Theorem~\ref{thm:integrality}.}
By ${\cal P}^d(A) = \Conv \bigcup_{\ell}{\cal P}^{d}_{\ell}(A)$ 
and ${\cal P}^{\leq d}(A) = \Conv \bigcup_{f} {\cal P}^f(A)$,
it suffices to show that ${\cal P}_{\ell}^d(A)$ is integral.
Since 
\begin{equation*} 
f_{{\cal P}_{\ell}^d(A)}(c) 
= \frac{1}{d} \max_{I,J \subseteq [n]: |I|=|J|=\ell } \deg \det (A[c])[I,J]^{\{d\}},
\end{equation*}
by Lemma~\ref{lem:EdmondsGiles}
it suffices to show $\deg \det (A[c])[I,J]^{\{d\}} \in d \ZZ \cup \{-\infty\}$. This follows from
\begin{Lem}\label{lem:regularity_B} 
For matrix $B$ in (\ref{eqn:B}), it holds
$\deg \det B^{\{d\}} \in d \ZZ \cup \{-\infty\}$.
\end{Lem}

We extend 
the proof~\cite{DerksenMakam2018} of the regularity lemma by the generic division algebra.
For a field $\FF$, $d \geq 2$, and $m \geq 2$, 
the generic matrix ring $R_d = R_d(\FF)$ is 
an $\FF$-algebra generated 
by $d \times d$ indeterminate matrices $X_k = (x_{k,ij})$ for $k =1,2,\ldots,m$. 
Let $Z(R_d)$ denote the center of $R_d$ (the set of elements commutative with every element of $R_d$).
The {\em generic division algebra (universal division algebra)} $UD_d = UD_d(\FF)$ is defined as
the ring of the fraction of $R_d$ by $Z(R_d) \setminus \{0\}$. 
Namely, $UD_d$ consists of $F/G$ for $F \in R_d$ and $G \in Z(R_d) \setminus \{0\}$, where addition and multiplication are defined 
in a natural way. 
It is known (see e.g., \cite[Section 3.2]{Rowen1980} and \cite[Chapter B.12-15]{DrenskyFormakek2004}) that $UD_d$ is a skew field.
Further, any element of $Z(R_d)$ is the scalar matrix $f I_d$ for some polynomial $f$ on variables $x_{k,ij}$. 
Therefore, by injective homomorphism $UD_d \ni F/(fI_d) \mapsto F/f \in 
\FF(\{x_{k,ij}\})^{d \times d}$, 
we can regard $UD_d \subseteq \FF(\{x_{k,ij}\})^{d \times d}$; see \cite[Proof of Proposition 2.1]{Kaliuzhnyi-VerbovetskyiVinnikov2012}. 
Suppose that $\FF = \KK(t)$.
Then $S = F/f \in UD_d(\KK(t))$ is written as 
$(1/f) \sum_{-\infty < \alpha \leq \alpha^*} t^{\alpha} F_{\alpha}$, where $F_{\alpha} \in R_d(\KK)$ and $F_{\alpha^*} \neq O$.
Then the degree $\deg S$ of $S$ is defined as $\alpha^* - \deg f \in \ZZ$.
On the other hand, since $S$ is a $d \times d$ matrix 
over $(\KK(t))(\{x_{k,ij}\})$, 
we can also consider $\det S$ and its degree $\deg \det S$.
\begin{Lem}\label{lem:deg_S}
For $S \in UD_d(\KK(t))$, it holds
$\displaystyle \deg S = \frac{1}{d} \deg \det S$.
\end{Lem}
\begin{proof}
In the above expression $S=(1/f)\sum_{-\infty < \alpha \leq \alpha^*} t^{\alpha} F_{\alpha}$, 
since $F_{\alpha^*} (\neq O)$ is invertible in 
$UD_d(\KK) \subseteq \KK(\{x_{k,ij}\})^{d \times d}$, 
it must hold $\det F_{\alpha^*} \neq 0$.
This implies $\deg \det S = d (\alpha^* - \deg f) \in d \ZZ$.
\end{proof}
For an $n \times n$ matrix $M$ over $UD_d(\KK(t))$, we can consider 
the Dieudonn\'e determinant of $M$, 
which is denoted by $\Det_{d} M$.  
Since $\deg$ is an abelian valuation on $UD_d(\KK(t))$, 
the degree $\deg \Det_{d} M$ is well-defined.
On the other hand, since $M$ is an $nd \times nd$ matrix over 
$(\KK(t))(\{x_{k,ij}\})$, we can also consider $\det M$ and its degree $\deg \det M$.
\begin{Lem}\label{lem:degDet_d_M}
For $M \in UD_{d}(\KK(t))^{n \times n}$, 
it holds 
$\deg \Det_d M = \displaystyle \frac{1}{d} \deg \det M$.
\end{Lem}
\begin{proof}
Consider the Bruhat decomposition $M = LPDU$ of $M$, 
where $L$ and $U$ are lower and upper triangular matrices, 
respectively, with ones in diagonals, $P$ is a permutation matrix,  $D$ is a diagonal matrix, and $PD$ is uniquely determined; see \cite[Theorem 9.2.2]{Cohn_Algebra}.
Suppose that $M$ is nonsingular (invertible). Then
the Dieudonn\'e determinant of $M$ is defined as the sign of $P$ 
times the product of the diagonals $D_{ii} (\neq O)$ of $D$ 
modulo the commutator subgroup of $UD_d(\KK(t))\setminus \{0\}$.
The degree $\deg M \in \ZZ$ is given by 
$\deg \prod_{i \in [n]} D_{ii} = \sum_{i \in [n]} \deg D_{ii}$ (independent of the order of multiplication).
View $M = LPDU$ as the product 
of $nd \times nd$ matrices $L, P, D, U$.
Then, $L, U$ are triangular matrices with ones in diagonals 
and $P$ is a permutation matrix, and 
$D$ is a block-diagonal matrix with $d \times d$ diagonal blocks  $D_{ii}$. Then, by Lemma~\ref{lem:deg_S}, it holds $\deg \det M = 
\deg \prod_{i =1}^n \det D_{ii} 
= \sum_{i =1}^n \deg \det D_{ii} = d \sum_{i=1}^n 
\deg D_{ii} = d \deg \Det_d M  \in d\ZZ$. 
If $M$ is singular, then $\deg \Det_d M$ is defined as $-\infty$. 
In this case, some $D_{ii}$ is zero matrix, 
and $\deg \det M = - \infty$.
\end{proof}

\begin{proof}[Proof of Lemma~\ref{lem:regularity_B}]
$B^{\{d\}}$ 
is viewed as an $n \times n$ matrix over $UD_d(\KK(t))$ whose $(i,j)$-element 
is given by $\sum_{k=1}^m (B_k)_{ij}X_k$. By Lemma~\ref{lem:degDet_d_M}, it holds $\deg \det B^{\{d\}} = d \deg \Det_d B^{\{d\}} \in d \ZZ \cup \{-\infty\}$.
\end{proof}

\subsection{Integrality gap}
We show that the integrality gap of ${\cal Q}(A)$ relative to ${\cal P
}(A)$
is at least $1/2$, which generalizes  Theorem~\ref{thm:2-approx_nc}.
\begin{Thm}\label{thm:2-approx_poly}
\begin{equation*}
\frac{1}{2} {\cal Q}(A) \subseteq {\cal P}(A).
\end{equation*}
\end{Thm}
By Lemma~\ref{lem:P_subseteq_Q}, 
this result follows from the next theorem with 
$B = A[c]$.
\begin{Thm}\label{thm:2-approx}
For a matrix $B$ in (\ref{eqn:B}), it holds that
\begin{equation*}
\frac{1}{2} \varDelta_{\rm max}(B) \leq  \delta_{\rm max}(B). 
\end{equation*}
\end{Thm}
\paragraph{Proof of Theorem~\ref{thm:2-approx}.}
We adapt  
the algorithmic proof of Theorem~\ref{thm:2-approx_nc} given in~\cite[Section 3]{BlaserJindalPandey2018}. 
For $\ell\in [n]$, define the maximum degree of subdeterminants of $B$ with size~$\ell$:
\begin{equation*}
\delta_\ell(B):= \max_{I,J \subseteq [n]: |I|=|J|=\ell} \deg \det B[I,J].
\end{equation*}
For $\xi \in \KK^m$, 
let $B[\xi] := \sum_{k=1}^m B_k \xi_k \in \KK(t)^{n \times n}$ 
be the matrix obtained from $B$ by substituting $\xi_k$ to $x_k$.
For $k \in [m]$, Let ${\bf 1}_k \in \RR^m$ denote the $k$-th unit vector. 
Recall the Smith-McMillan form for a matrix over $\KK(t)$; see e.g., \cite[Section 5.1.2]{MurotaMatrix}.
\begin{Lem}\label{lem:SM}
Let $|\KK| > n$ and $\xi \in \KK^m$.
Suppose that
\begin{equation}\label{eqn:suppose}
\delta_{\ell} (B[\xi]) \geq \delta_{\ell} (B[\xi+ s {\bf 1}_k]) \quad (\ell \in [n], k \in [m], s \in \KK). 
\end{equation}
If the Smith-McMillan form of $B[\xi]$ is given by
\begin{equation*}
SB[\xi]T = (t^{\alpha}),
\end{equation*}
for $S,T \in GL_n(\KK(t)^-)$ and 
$\alpha \in (\ZZ \cup \{- \infty\})^n$ ordered as
\begin{equation*}
\alpha_1 \geq \alpha_2 \geq \cdots  \geq \alpha_{r} > \alpha_{r+1}= \cdots = \alpha_n = -\infty,
\end{equation*}
then it holds
\begin{equation}\label{eqn:SB_kT}
\deg (SB_kT)_{ij} \leq \alpha_i \geq  \deg (SB_kT)_{ji} \quad (i,j \in [n]:i \leq j).
\end{equation}
\end{Lem}
\begin{proof}
Let $C_{\xi,s} := SB[\xi+ s {\bf 1}_k]T$ and $C_k := SB_kT$.
By $B[\xi+ s {\bf 1}_k] = B[\xi] + s B_k$, it holds
$C_{\xi,s} = (t^{\alpha}) + s C_k.$
Let $I,J \subseteq [n]$ with $|I| = |J| =: h$.
View $\det C_{\xi,s}[I,J]$ as a polynomial on $s$:
\begin{equation*}
\det C_{\xi,s}[I,J] =b_0 + b_1 s + b_2 s^2 + \cdots + b_h s^h,
\end{equation*}
where $b_i = b_i(t)$ is a rational function on $t$.
By (\ref{eqn:suppose}), it holds
\begin{equation*}
 \deg \det C_{\xi,s}[I,J] \leq \delta_h (C_{\xi,s}) = \delta_h (B[\xi+s {\bf 1}_k]) \leq \delta_h (B[\xi])  = \delta_h((t^{\alpha})) = \sum_{\ell=1}^h \alpha_\ell, 
\end{equation*}
where the first and second equalities follow from a basic fact 
that the maximum degree is invariant under 
multiplication by $GL_n(\KK(t)^-)$; see e.g., \cite[Theorem 5.1.5]{MurotaMatrix}.
Since $|\KK| > n \geq h$, it must hold
\begin{equation}\label{eqn:b_i}
\deg b_i \leq \sum_{\ell=1}^h \alpha_\ell \quad (i = 0,1,\ldots,h).
\end{equation}

We show (\ref{eqn:SB_kT}) by focusing on $b_1$ whose (nonzero) term is the product of one entry of $C_k$ and
$h-1$ diagonals of $(t^{\alpha})$.
Suppose that $i,j > r$.
Let $I := [r] \cup \{i\}$, and $J := [r] \cup \{j\}$. 
Then it holds
$
b_1 =  (C_{k})_{ij} t^{\sum_{\ell=1}^r \alpha_\ell}.
$
By (\ref{eqn:b_i}) and $\delta_{h}(B[\xi]) = -\infty$ for $h > r$,  
we have $(C_k)_{ij} = (SB_kT)_{ij} = 0$ and (\ref{eqn:SB_kT}).
Suppose that $i \leq r$ and $i < j$.
Let $I := [i]$ and $J := [i-1] \cup \{j\}$.
Then we have
$
b_1 =  (C_k)_{ij}  t^{\sum_{\ell =1}^{i-1} \alpha_\ell}$.
By (\ref{eqn:b_i}), 
we have $\deg (SB_kT)_{ij} = \deg (C_k)_{ij} \leq \alpha_i$.
The case $i > j$ is similar.
Suppose that $i=j \leq r$.
Let $I = J := [i]$. Then it holds 
$
b_1 = \sum_{j=1}^i  (C_k)_{jj}
t^{\sum_{\ell\in [i] \setminus \{j\}} \alpha_\ell}. 
$
We show by induction on $i$ that $\alpha_i \geq \deg (C_k)_{ii}$. 
The case $i=1$ is clear from (\ref{eqn:b_i}). 
Suppose that $\alpha_i < \deg(C_k)_{ii}$. 
Then the term $(C_k)_{ii} t^{\sum_{\ell=1}^{i-1} \alpha_\ell}$
has degree greater than $\sum_{\ell=1}^i \alpha_\ell$, and 
must cancel out with other terms. 
This is impossible. Indeed, 
by induction, it holds $(C_k)_{jj} \leq \alpha_j$ for $j < i$, and 
any other term has degree at most $\sum_{\ell=1}^i \alpha_\ell$. 
Thus we have (\ref{eqn:SB_kT}) for $i=j$.
\end{proof}

Since $\deg \det$ and $\deg \Det$ are invariant under any field extension of $\KK$ (by (\ref{eqn:degDet_formula1})),
we can assume that $\KK$ is an infinite field. 
Then we can choose ``generic" $\xi \in \KK^m$ satisfying (\ref{eqn:suppose}).
Consider the Smith-McMillan form $SB[\xi]T = (t^{\alpha})$ as in Lemma~\ref{lem:SM}.
Define $\alpha^* \in \ZZ^n_{\geq 0}$ by $\alpha^*_{i} := \max (0, \alpha_i)$. 
By (\ref{eqn:SB_kT}), it holds
$\deg ((t^{-\alpha^*}) SB_kT(t^{-\alpha^*}))_{ij} \leq 0$ for $i,j \in [n], k \in [m]$.
This means that $p = q = \alpha^*, S, T$ satisfies 
the condition of the minimization problem in~(\ref{eqn:A[c]_dual}).
Thus 
$
\Delta_{\rm max}(B) \leq 2 \sum_{\ell=1}^{r^*} \alpha_\ell = 2 \delta_{\rm max}((t^{\alpha})) = 2 \delta_{\rm max}(B[\xi]) \leq  2 \delta_{\max}(B).$

\subsection{Rank-2 skew-symmetric matrices}\label{subsec:rank2-skew-symmetric}
Here we focus on the case where each $A_k$ is a rank-2 skew-symmetric matrix. 
Instead of $A_k$, we consider the 2-dimensional subspace $\ell_k$ 
since $A_k$ is written, up to nonzero multiple, as $u_kv_k^{\top} - v_ku_k^{\top}$
for any basis $u_k,v_k$ of $\ell_k$.
Therefore, Edmonds' problem is also formulated as: 
Given a collection $L = \{\ell_1,\ell_2,\ldots,\ell_m\}$
of 2-dimensional subspaces in $\KK^n$, compute the rank of 
a matrix $A_L$ defined by
\begin{equation}\label{eqn:A_L}
A_L := \sum_{k =1}^m (u_kv_k^{\top} - v_ku_k^{\top}) x_k,
\end{equation}
where $u_k,v_k$ is a basis of $\ell_k$.
As Lov\'asz~\cite{Lovasz89} pointed out, 
Edmonds' problem for $A_L$ is the matroid matching problem 
for the linear matroid represented by $u_1, v_1, u_2, v_2,\ldots,u_m, v_m$.
Here a {\em (matroid) matching} for $L$ is an index subset $I \subseteq [m]$ 
such that $\dim \sum_{k \in I} \ell_k = \sum_{k \in I} \dim \ell_k = 2|I|$ or equivalently $u_k,v_k$ $(k \in I)$ are linearly independent.
\begin{Lem}[\cite{Lovasz89}]
\begin{equation}\label{eqn:matroid_matching}
\rk A_L = 2 \max \{ |I| \mid I: \mbox{matching for $L$}\}. 
\end{equation}
\end{Lem}
Oki and Soma~\cite{OkiSoma_SICOMP} showed that 
noncommutative Edmonds' problem for $A_L$ corresponds to 
the fractional matroid matching problem (Vande Vate~\cite{VandeVate92}), 
where the underlying matroid is the infinite matroid of all vectors in $\KK^n$.
A {\em fractional (matroid) matching} here is a vector $u \in \RR^{m}_{\geq 0}$ 
satisfying $\sum_{k=1}^m u_k \dim \ell_k \cap X \leq \dim X$ 
for all vector subspaces $X$ in $\KK^n$.
\begin{Thm}[\cite{OkiSoma_SICOMP}]\label{thm:OkiSoma}
\begin{equation}\label{eqn:fractional_matching}
\ncrk A_L =  2 \max \{ {\bf 1}^{\top} u \mid u: \mbox{fractional matching for $L$} \}.  
\end{equation}
\end{Thm}
They also showed that the nc-rank of $A_L$ is attained by the second blow-up.
\begin{Thm}[\cite{OkiSoma_SICOMP}]\label{thm:2nd_blowup_OkiSoma}
\begin{equation*}
\ncrank A_L =  \frac{1}{2} \rank A_L^{\{2\}}.  
\end{equation*}
\end{Thm}
We consider these results in our polyhedral setting.
The {\em matroid matching polytope} ${\rm MP}(L)$ 
is defined as the convex hull of the characteristic vectors 
${\bf 1}_{I}:= \sum_{k \in I}{\bf 1}_k$ 
over all matchings $I$ for $L$.
Then (\ref{eqn:matroid_matching}) is linear optimization
on ${\rm MP}(L)$ with objective vector~${\bf 1}$. 
As is expected, ${\rm MP}(L)$ equals the half of ${\cal P}(A_L)$. 
\begin{Lem}\label{lem:P(A_L)=2M(L)}
\[
{\cal P}(A_{L}) = 2 {\rm MP}(L).
\]
\end{Lem}
This fact appears in \cite{HIOS2025}.
For completeness, we give a proof in the appendix.
In particular, $c \mapsto \deg \det A_L [c]$ 
is weighted maximization on (the support function for) ${\rm MP}(L)$.
Extending Theorem~\ref{thm:OkiSoma}, 
the paper~\cite{HIOS2025} showed the corresponding relationship between $c \mapsto \deg \Det A_L [c]$ and the {\em fractional matroid matching polytope} ${\rm FMP}(L)$---the polytope consisting of all fractional matroid matchings.
\begin{Thm}[\cite{HIOS2025}]
\[
{\cal Q}(A_L) = 2 {\rm FMP}(L).
\]
\end{Thm}

We present two results on ${\cal Q}(A_L)$.
The first one is an expected generalization of Theorem~\ref{thm:2nd_blowup_OkiSoma}.
\begin{Thm}\label{thm:2nd_blowup}
\[
{\cal P}^{2}(A_L) = {\cal P}^{\leq 2}(A_L) = {\cal Q}(A_L).
\]
\end{Thm}
Since linear optimization on ${\cal P}(A_L)$ is solved in polynomial-time~\cite{IwataKobayashi2017}, 
this may support our expectation that 
small $r$ with ${\cal P}^{\leq r}(A) = {\cal Q}(A)$ implies tractability of Edmonds' problem.

The second one is a stronger estimate of the integrality gap of ${\cal Q}(A_L)$.
We have seen in Theorem~\ref{thm:2-approx_poly} that $(1/2){\cal Q}(A_L) \subseteq {\cal P}(A_L)$, which is improved to:
\begin{Thm}\label{thm:3/2-approx}
\[
\frac{2}{3}{\cal Q}(A_L) \subseteq {\cal P}(A_L).
\]
\end{Thm}
In particular, the fractional matroid matching relaxation has integrality gap $2/3$.
This is a generalization of the classical fact that a natural fractional relaxation ({\em fractional matchings}) of matchings in undirected graphs 
has integrality gap $2/3$.
\begin{Rem}
    Lee, Sviridenko and Vondr\'ak~\cite{LeeSviridenkoVondrak2013} showed that fractional matroid matching relaxation 
    has unbounded integrality gap, even for linear matroids. See also \cite{KaparisLetchfordMourtos2020} for related results.
    Our Theorem~\ref{thm:3/2-approx} does not contradict these results.
In fact, fractional matroid matchings are defined for a general matroid $\bf M$ together with rank-$2$ flats $\ell_1,\ell_2,\ldots,\ell_m$.
By using the lattice $({\cal L}_{\bf M},\wedge,\vee)$ of flats and the rank function $r$, fractional matroid matchings with respect to ${\bf M}$ are defined 
as vectors $u \in \RR_{\geq 0}^m$ satisfying $\sum_{k =1}^m r(X \wedge \ell_k) u_k \leq r(X)$ for all flats $X \in {\cal L}_{\bf M}$.
Their fractional matroid matching is with respect to the (finite) matroid ${\bf M}$ represented by vectors
$u_1,v_1,u_2,v_2,\ldots,u_m,v_m$ in (\ref{eqn:A_L}).
On the other hand, our fractional matroid matching is with respect to the infinite matroid ${\bf M}(\KK^n)$ of all vectors in $\KK^n$. Since $r(X \wedge \ell_k) \leq \dim X \cap  \ell_k$ and ${\cal L}_{\bf M} \subset {\cal L}_{{\bf M}(\KK^n)}$, our relaxation is stronger than theirs. See also Remark~\ref{rem:2/3}. 
\end{Rem}

\paragraph{Proof of Theorem~\ref{thm:2nd_blowup}.}
By specialization $X_k \leftarrow x_k I_2$, 
the second blow-up $A_L^{\{2\}}$ becomes $A_L \otimes I_2$, 
from which we see ${\cal P}^{2}(A_L) = {\cal P}^{\leq 2}(A_L)$.
For showing ${\cal P}^{2}(A_L) = {\cal Q}(A_L)$, 
we prove that for all $c \in \ZZ^m$ it holds
\begin{equation}\label{eqn:goal} 
\varDelta_{\rm max}(A_L[c]) \leq \frac{1}{2} \max_{I,J \subseteq [n]: |I|=|J|} \deg \det (A_L[c])[I,J]^{\{2\}}\ (=  f_{{\cal P}^{2}(A_L)}(c)).
\end{equation}
Let $B := A_L[c]$ for simplicity. We may assume that $c$ is an even vector in $(2\ZZ)^m$ (by positive homogeneity of support functions). Then $p,q$ in (\ref{eqn:A[c]_dual}) can be assumed to even vectors.
By skew-symmetry of $A_L$ and \cite[Lemma 3.11]{HIOS2025}, 
the minimum in (\ref{eqn:A[c]_dual}) is attained by $T = S^{\top}$ 
and $p = q \in \ZZ^n_{\geq 0}$. 
Consider $S \in GL_n(\KK)$ and $p \in \ZZ^n_{\geq 0}$ 
with $\deg (t^{-p})SBS^{\top}(t^{-p})_{ij} \leq 0$ for $i,j \in [n]$.
As in the ordinary determinant, 
it holds $\deg \Det LM = \deg \Det L+\deg \Det M$ 
and $\deg M_{ij} \leq 0 \ (i,j \in [n]) \Rightarrow \deg \Det M \leq 0$; see \cite[Lemmas 2.8 and 2.10]{HH_degdet}. Therefore we have  
\begin{equation*}
    0 \geq \deg \Det ((t^{-p})SBS^{\top}(t^{-p}))[I,J] = - \sum_{i\in I}p_i - \sum_{j\in J}p_j + \deg \Det (SBS^{\top})[I,J]. 
\end{equation*}
Therefore, we have
\begin{equation}\label{eqn:degDet<=21^Tp}
\deg \Det (SBS^{\top})[I,J] \leq \sum_{i\in I}p_i + \sum_{j\in J}p_j \leq 2{\bf 1}^{\top}p,
\end{equation}
Taking the maximum over indices $I,J \subseteq [n]$, we have
\begin{equation}\label{eqn:Delta<=21^Tp}
\varDelta_{\rm max}(B)=\varDelta_{\rm max}( SBS^{\top}) \leq 2{\bf 1}^{\top}p, 
\end{equation}
where the equality follows from the fact that 
the maximum degree of Dieudonn\'e deteminants is also invariant
under multiplication by $GL_n(\KK(t)^-)$; see \cite[Proposition 29 and Section A.4]{HH_degdet}.
Suppose that $p,S$ attains the minimum. 
Then the equality in (\ref{eqn:Delta<=21^Tp}) holds.
Necessarily,
all equalities in (\ref{eqn:degDet<=21^Tp}) hold for some $I,J \subseteq [n]$.
Such an index pair is exactly $(I,J)$ satisfying
\begin{itemize}
\item[($*$)] $\deg \Det ((t^{-p})SBS^{\top}(t^{-p}))[I,J] = 0$ and $I \cap J \supseteq \{i \in [n] \mid p_i > 0\}$. 
\end{itemize}
\begin{Clm}
There is $K \subseteq [n]$ such that $(K,K)$ satisfies ($*$).
\end{Clm}
For the rest of the proof, we use the expression
\begin{equation}\label{eqn:expand}
(t^{-p})SBS^{\top}(t^{-p}) = \sum_{k=1}^m C_kx_k  +  t^{-1}Q,
\end{equation}
where $C_k$ is a skew-symmetric matrix over $\KK$ and $Q \in (\KK(x_1,\ldots,x_m)(t)^{-})^{n\times n}$.
Necessarily, it holds $\rank C_k \leq 2$.
Indeed, we may assume that $-p$ is a decreasing vector.
Then $C_k$ is an anti-block diagonal matrix, 
and $(t^{-p})SBS^{\top}(t^{-p})$ is a lower block anti-triangular matrix.
From this, we see that $\rank C_k > 2$ would imply a contradiction 
$\rank S(u_kv_k^{\top} - v_ku_k^{\top})S^{\top} > 2$.
Let $C := \sum_{k=1}^m C_kx_k$.
\begin{proof}[Proof of Claim] 
As in the ordinary determinant,
$\deg \Det ((t^{-p})SBS^{\top}(t^{-p}))[I,J] = 0$ if and only if $C[I,J]$ is nc-nonsingular (i.e., $\ncrk C[I,J] = |I|$); see \cite[Lemma 2.10]{HH_degdet}.
In general, for a matrix $M$ over a skew field $\FF$, 
the family of all pairs $I,J$ of indices with $M[I,J]$ nonsingular 
forms a {\em bimatroid}~\cite{Kung1978}\footnote{The (right) linear independence of the columns of $[M\ I]$ has the matroid property even if $\FF$ is noncommutative. 
From this, the bimatroid property of $M$ follows.}; see also \cite[Section 2.3.7]{MurotaMatrix}. 
Then, it satisfies the exchange property: 
If $M[I,J]$ and $M[I',J']$ are nonsingular and $i' \in I' \setminus I$, 
then there is $j' \in J'\setminus J$ 
such that $M[I \cup \{i'\},J \cup \{j'\}]$ is nonsingular or there is $i \in I \setminus I'$ such that
$M[I \cup \{i'\} \setminus \{i\},J]$ is nonsingular.

Consider a pair $(I,J)$ of indices satisfying ($*$) such that $|I \cap J|$ is maximum.  
Suppose (for contradiction) that $I \neq J$.
Consider the bimatroid for $C$.
By skew-symmetry, it holds $(C[I,J])^{\top} = - C[J,I]$.
Now $C[I,J]$ is nc-nonsingular. 
By Theorem~\ref{thm:FortinReutenauer}, its transpose $(C[I,J])^{\top}$ is also nc-nonsingular.\footnote{On a general skew field, 
transpose $M^{\top}$ of a nonsingular matrix $M$ is not necessarily nonsingular.}
Therefore $(J,I)$ also satisfies ($*$). 
Apply the exchange property for $I,J, I'=J, J' = I$ and $j \in J\setminus I$.
There is $i \in I \setminus J$ 
such that $C[I \cup \{j\},J \cup \{i\}]$ is nc-nonsingular or
$C[I \cup \{j\} \setminus \{i\},J]$ is nc-nonsingular.
Then, both $(I \cup \{j\},J \cup \{i\})$ and 
$(I \cup \{j\} \setminus \{i\},J)$ satisfy ($*$), with contradicting   
the choice of $(I,J)$. 
Thus we have $I=J =:K$.  
\end{proof}


Choose such an index set $K$ in the claim.
Now $C[K,K] (= \sum_{k=1}^m C_k[K,K]x_k)$ is nc-nonsingular and 
each $C_k[K,K]$ is rank-2 skew-symmetric (as long as $C_k[K,K] \neq O$).
By Theorem~\ref{thm:2nd_blowup_OkiSoma}, the second blow-up $C[K,K]^{\{2\}}$ is nonsingular. 
Consider the second blow-up of (\ref{eqn:expand}):
\begin{equation*}
((t^{-p})SBS^{\top}(t^{-p}))[K,K]^{\{2\}} = C[K,K]^{\{2\}}  +  t^{-1}Q[K,K]^{\{2\}},
\end{equation*}
Then the degree of the determinant of the LHS is zero:
\begin{equation*}
0=  \deg \det ((t^{-p})SBS^{\top}(t^{-p}))[K,K]^{\{2\}} = - 4 {\bf 1}^{\top} p 
+ \deg \det (SBS^{\top})[K,K]^{\{2\}}. 
\end{equation*}
Here, $B^{\{2\}}$ can be viewed as an $n \times n$ matrix over $UD_2(\KK(t))$.
Then $(SBS^{\top})[K,K]^{\{2\}}$ is written as $(SB^{\{2\}}S^{\top})[K,K]$. By Lemma~\ref{lem:degDet_d_M}, we have (\ref{eqn:goal}):
\begin{eqnarray*}
 &&\varDelta_{\rm max}(B) = 2 {\bf 1}^{\top} p = \deg \Det_2 (SB^{\{2\}}S^{\top})[K,K]  \leq \max_{I,J} 
\deg \Det_2  (SB^{\{2\}}S^{\top})[I,J]   \nonumber \\
&&=  \max_{I,J}\deg \Det_2 B^{\{2\}} [I, J]  
=    \frac{1}{2} \max_{I,J} \deg \det B[I,J]^{\{2\}},
\end{eqnarray*}
where the third equality follows from the same reason we used in (\ref{eqn:Delta<=21^Tp}).

\paragraph{Proof of Theorem~\ref{thm:3/2-approx}.}
The proof combines 
the classical $2/3$-approximation of weighted matching and 
characterizations of (fractional) matroid matchings using flags of vector subspaces.
Here a {\em complete flag} ${\cal X} = (X_i)$ of $\KK^n$
is an ordered tuple of vector subspaces $X_0,X_1,\ldots,X_n$ such that
\begin{equation*}
\{0\} = X_0 \subset X_1 \subset X_2 \subset \cdots \subset X_n = \KK^n,
\end{equation*}
where $\dim X_i = i$.
Let $L = \{\ell_1,\ldots,\ell_m\}$ be a collection of 2-dimensional subspaces.
Let ${\cal X} = (X_i)$ be a complete flag.
For each $k \in [m]$, there are exactly two indices 
$i$ with $\dim \ell_k \cap X_{i} - \dim \ell_k \cap X_{i-1} = 1$. Regard this pair of indices 
as an edge $e^{\cal X}_k$ of the complete graph $K_n$ on vertex set $[n]$.
For $I \subseteq [m]$, define edge set $I^{\cal X} := \{e_k^{\cal X} \mid k \in I\}$ (which is allowed be a multiset).


\begin{Lem}\label{lem:flag_matching}
For $I \subseteq [m]$,
the following conditions are equivalent:
\begin{itemize}
\item[(1)] $I$ is a matching for $L$.
\item[(2)] There is a complete flag ${\cal X} = (X_i)$ of $\KK^n$ 
such that $I^{\cal X}$ is a matching in $K_n$ 
\end{itemize}
\end{Lem}
\begin{proof}
(1) $\Rightarrow$ (2). We can choose a basis $f_1,f_2,\ldots,f_n$ of $\KK^n$ 
such that each $\ell_k$ is spanned by some pair in the basis.
Then, (2) is satisfied by the complete flag ${\cal X} = (X_i)$ such that $X_i$ is spanned by $f_1,f_2,\ldots,f_i$.

(2) $\Rightarrow$ (1).
By coordinate transformation, 
we may assume that ${\cal X}$ 
is a coordinate flag, that is, 
$X_i$ is spanned by 
${\bf 1}_1,{\bf 1}_2,\ldots,{\bf 1}_i$.
For each $k \in I$, if  $e_k^{\cal X} = ij$ 
with $i < j$, 
then we choose a basis $u_k,v_k$ of $\ell_k$ 
so that $u_k \in X_i$ with $(u_k)_i \neq 0$ 
and $v_k \in X_j$ with $(v_k)_j \neq 0$.  
Since $I^{\cal X}$ is a matching, 
$2|I|$ vectors $u_k, v_k$ are distinct column vectors
of an $n \times n$ upper-triangular matrix with nonzero diagonals.
This means that $u_k, v_k$ $(k \in I)$ are linearly independent, and $I$ is a matching for $L$.
%
%
\end{proof}

We next consider a characterization of extreme points of ${\rm FMP}(L)$ due to Chang, Llewellyn, and Vande Vate~\cite{CLV01b}.
Note that ${\rm FMP}(L) (= (1/2){\cal Q}(A_L))$ is a half-integral polytope belonging to $[0,1]^m$. 
Therefore, any extreme point is a vector in $\{0,1/2,1\}^m$.
For $u \in \{0,1/2,1\}^{m}$, 
let $I_{1/2}(u) :=\{k \in [m] \mid u_k = 1/2\}$, $I_{1}(u) :=\{k \in [m] \mid u_k = 1\}$, and $I(u) := I_1(u) \cup I_{1/2}(u)$.
\begin{Thm}[\cite{CLV01b}]\label{thm:CLV_characterization}
For $u \in \{0,1/2,1\}^{m}$, the following conditions are equivalent:
\begin{itemize}
\item[(1)] $u$ is an extreme point of ${\rm FMP}(L)$. 
\item[(2)] There is a complete flag ${\cal X} = (X_i)$ of $\KK^n$ satisfying the following:
\begin{itemize}
\item[(2-0)]  $X_{2 {\bf 1}^{\top}u } = \sum_{k \in I(u)}\ell_k$.
\item[(2-1)] $\sum_{k=1}^m u_k \dim \ell_k \cap X_i = \dim X_i$ for $i \in [2{\bf 1}^{\top}u]$. 
\item[(2-2)] $X_{2|I_1(u)|} = \sum_{k \in I_1(u)}\ell_k$; in particular, $I_1(u)$ is a matching for $L$.
\item[(2-3)] $I_1(u)^{\cal X}$ is a matching,  
$I_{1/2}(u)^{\cal X}$ consists of odd cycles, 
and $I(u)^{\cal X}$ is their node-disjoint union.
\end{itemize}
\end{itemize}
\end{Thm}
The original statement in \cite{CLV01b} is about intermediate flag $X_{2|I_1(u)|}\subset \cdots \subset X_{2{\bf 1}^{\top} u}$.
We complement it to a complete flag (by using the previous lemma).
Note that (2-1) is redundant.

For showing $(2/3)\varDelta_{\rm max}(A_L[c]) \leq \delta_{\rm max}(A_L[c])$, 
we construct a $2/3$-approximate matching
from an extreme point $u^* \in {\rm FMP}(L)$ 
attaining the maximum $c^{\top}u$ over $u \in {\rm FMP}(L)$.
It necessarily holds $c_k \geq 0$ if $k \in I(u^*)$.
Consider a complete flag ${\cal X}$ in Theorem~\ref{thm:CLV_characterization}. 
For simplicity, we do not distinguish between $I(u^*)$ and $I(u^*)^{\cal X}$, and we denote $c^{\top}{\bf 1}_{S}$ by~$c(S)$.

Let $C_1,C_2,\ldots,C_s$ be the set of odd cycles of $I_{1/2}(u^*)^{\cal X}$.
By construction, $I(u^*)^{\cal X}$ has no loop, and $|C_i| \geq 3$. 
Then, each $C_i$ is written as the union of $|C_i|$ matchings $M^1_i,M^2_i,\ldots,M_i^{|C_i|}$
of size $(|C_i|-1)/2$ 
such that each edge in $C_i$ is used by these matchings exactly $(|C_i|-1)/2$ times. Namely, it holds
${\bf 1}_{C_i} = (2/(|C_i|-1)) \sum_{j=1}^{|C_i|} {\bf 1}_{M^j_i}. $
For each $i \in [s]$, 
choose $M_i^*$ from  $M_i^1,M_i^2,\ldots,M_i^{|C_i|}$ such that 
 $c(M_i^*) = \max_{j=1,2,\ldots,|C_i|} c(M_i^j)$. 
Then $M^* := I_1(u^*) \cup M^*_1 \cup \cdots \cup M^*_s$
forms a matching in $K_n$, 
and is a matching for $L$ by Lemma~\ref{lem:flag_matching}.
Therefore we have
\begin{eqnarray*}
&& c^{\top}u^* = c( I_1(u^*) ) 
+ \frac{1}{2} \sum_{i=1}^s 
c(C_i)=  c( I_1(u^*) )  + \frac{1}{2} \sum_{i=1}^s \frac{2}{|C_i|-1} \sum_{j=1}^{|C_i|} c(M_i^{j}) \\
&&\leq  c( I_1(u^*) ) + \sum_{i=1}^s \frac{|C_i|}{|C_i|-1} c(M_i^{*})\leq  c( I_1(u^*) ) + \frac{3}{2} \sum_{i=1}^s c(M_i^{*}) \leq \frac{3}{2} c(M^*).
\end{eqnarray*}
Thus $\varDelta_{\rm max}(A_L[c]) 
= 2c^{\top}u^* \leq 3 c(M^*) \leq (3/2) \delta_{\rm max}(A_L[c])$, which completes the proof.
\begin{Rem}\label{rem:2/3}
This proof yields a $2/3$-approximation algorithm for weighted linear matroid matching, by combining a strongly polynomial-time algorithm 
for $\varDelta_{\rm max}(A[c])$ in~\cite{HIOS2025}.
We sketch it, though it may no longer be of much interest in the light of~\cite{IwataKobayashi2017}.
By perturbing weight $c\in \ZZ^m$ (as in \cite[(4.17)]{HiraiIkeda}), 
we can assume that an optimal fractional matching $u^*$ is unique, and that the dual LP of $\max \{c^{\top}u \mid  u\in {\rm FMP}(L)\}$ is nondegenerate. 
By the algorithm in~\cite{HIOS2025}, 
we obtain $p = q, S=T^{\top}$ attaining the minimum of~(\ref{eqn:A[c]_dual}). 
According to \cite[Section 4.1.2]{HIOS2025}, 
$(p,S)$ is transformed to
an optimal solution $\lambda$ of the dual LP whose nonzero support is a flag. By the nondegeneracy, we can make $\lambda$ an optimal basis solution with flag support ${\cal X}$. From the basis, we can determine the primal optimal solution $u^*$.
According to the proof of Theorem~\ref{thm:CLV_characterization} in~\cite[p.89]{CLV01b}, 
this flag ${\cal X}$ is transformed to the one in the theorem. 
Then, apply the above procedure for $u^*,{\cal X}$ and obtain a $2/3$-approximate matching $M^*$.

For general matroids, 
a $2/3$-approximation algorithm is known for restricted classes of the problems; see~\cite{Fujito,LeeSviridenkoVondrak2013}.
A larger part of the proof
can be formulated for general matroids by 
adding loop $e^{\cal X}_k = ii$ for each $\ell_k$ with $r(X_i \wedge \ell_k) - r(X_{i-1} \wedge \ell_k) = 2$ (that cannot occur in our setting). 
Then, the above odd cycle $C_i$ may be a loop, and we get stuck at the last step of the proof.
\end{Rem}
\section*{Acknowledgments}
The author thanks Yuya Goto, Yuni Iwamasa, Taihei Oki, 
and Tasuku Soma for helpful comments.
The author was supported 
by JSPS KAKENHI Grant Numbers JP24K21315, JP26H01996.
The author used Microsoft Copilot to assist with literature survey, organization of ideas, and language editing.
The author reviewed and revised all outputs and takes full responsibility for the content of the manuscript.

\bibliographystyle{plain}
\bibliography{polytope_Edmonds}

\appendix
\section{Proof of Lemma~\ref{lem:P(A_L)=2M(L)}}
For each $k \in [m]$, 
choose a basis $u_k,v_k$ of $\ell_k$. Consider $A_L$ in (\ref{eqn:A_L}). Let $B := A_L[c]$.
Consider $n \times 2m$ matrix $M := (u_1\ v_1\ u_2\ v_2\ \cdots\ u_m\ v_m)$. 
For $K \subseteq [n]$ and $s \in [m]$, 
let $N := M[K,[2s]]$, and define $2s \times 2s$ matrix $\Omega$ by 
\begin{equation*}
 \Omega_{ij} := \left\{
 \begin{array}{cl}
 x_{k}^{-1}t^{-c_k} & {\rm if}\ (i,j)=(2k,2k-1),k\in[s],\\
 -x_{k}^{-1}t^{-c_k} & {\rm if}\ (i,j)=(2k-1,2k),k\in [s],\\
 0 & {\rm otherwise},
 \end{array}
\right. \quad (i,j \in [2s]).
\end{equation*}
Let $B' := \sum_{k=1}^s (u_kv_k^{\top}- v_ku_k^{\top})x_kt^{c_k}$ be the matrix obtained 
from $B$ by letting $x_k := 0$ for $k > s$.
Our proof is based on the following relation:
\begin{equation}\label{eqn:relation}
\det B'[K,K]  t^{-2c{\bf 1}_{[s]}} \prod_{i=1}^s x_i^{-2} 
 = \det \left(
\begin{array}{cc}
B'[K,K] & N \\
O & \Omega
\end{array}
\right) = \det
\left(
\begin{array}{cc}
O & N \\
-N^{\top} & \Omega
\end{array}
\right),
\end{equation}
where the second equality is seen 
from column elementary operation.

We show $\delta_{\rm max}(B) = 2 \max \{c^{\top}u \mid u \in 
{\rm MP}(L)\}$ (by Lemma~\ref{lem:P_subseteq_Q}).
\begin{Clm}
    There is $K \subseteq [n]$ 
    such that $\deg \det B[K,K] = \delta_{\rm max}(B)$.
\end{Clm}
\begin{proof}
It is known that $I,J \mapsto v(I,J) := \deg \det B[I,J]$ is a {\em valuated bimatroid}~\cite[Section 5.2.5, Example 5.2.15]{MurotaMatrix}. Then the exchange property says that 
for $I,J, I',J' \subseteq [n]$ and $i' \in I' \setminus I$, there is $j' \in J'\setminus J$ 
such that 
\begin{equation*}
v(I,J) + v(I',J') \leq v(I\cup\{i'\},J\cup\{j'\}) + v(I' \setminus \{i'\},J'\setminus \{j'\}) 
\end{equation*}
or there is $i \in I \setminus I'$ such that
\begin{equation*}
v(I,J) + v(I',J') \leq v(I\cup\{i'\} \setminus \{i\},J) + v(I' \cup\{i\} \setminus \{i'\},J'). 
\end{equation*}
Choose $I,J$ with $v(I,J) = \delta_{\rm max}(B)$ 
such that $|I \cap J|$ is maximum. 
Suppose, to the contrary, that $I \neq J$.
By skew-symmetry, it holds $(B[I,J])^{\top} = - B[J,I]$, and $v(I,J) = v(J,I)$.
Apply the exchange property for $I,J, I' =J, J'=I$ and $j \in J \setminus I$.
There is $i \in I \setminus J$ such that $v(I \cup \{j\},J \cup \{i\}) = \delta_{\rm max}(B)$ 
or $v(I \cup \{j\} \setminus \{i\},J) = \delta_{\rm max}(B)$.
This is a contradiction to the choice of $I,J$. Thus we have $I = J =:K$.
\end{proof}
Choose such an index set $K \subseteq [n]$.
From the fact that each summand of $A_L$ is rank~2, we see that $\supp \det B[K,K] \subseteq \{0,1,2\}^m$.
Further, since $B$ is skew-symmetric,
$\det B[K,K]$ is the square of a polynomial (that is the Paffian).
Hence the maximum degree is attained 
by the term having even exponent vector $2{\bf 1}_{S}$ for $S \subseteq [m]$.
We may assume that $S = [s]$. Define $B'$, $N$, and $\Omega$ as above.
Then $\det B'[K,K] = a t^{2c^{\top}{\bf 1_{[s]}}} \prod_{i=1}^s x_i^2$ for $a \neq 0$.
In particular, (\ref{eqn:relation}) takes a nonzero value $a$ in $\KK$.
Here $|K| > 2s$ is impossible, since otherwise the right most determinant in (\ref{eqn:relation}) must be zero.
Also $|K| < 2s$ is impossible, since otherwise it must include some of $x_i$. 
Thus $a = (\det N)^2 \neq 0$.
This means that $u_k,v_k$ $(k \in [s])$ are linearly independent, and $[s]$ is a matching 
having weight $c^{\top}{\bf 1}_{[s]} = (1/2) \delta_{\rm max}(B)$.
Thus $\delta_{\rm max}(B) = 2c^{\top}{\bf 1}_{[s]} \leq 2 \max \{c^{\top}u \mid u \in {\rm MP}(L)\}$.

Consider the converse direction.
Suppose that the maximum-weight matching is given by $[s]$.
Then there is a row index set $K$ with $|K| = 2s$ 
such that $N := M[K, [2s]]$ is nonsingular. 
Define $B'$ and $\Omega$ as above.
Consider the relation~(\ref{eqn:relation}).
The right most determinant must be a nonzero value $a$ in $\KK$.
This implies $\deg \det B' = 2c{\bf 1}_{[s]}$. 
Necessarily we have  
$\delta_{\rm max}(B) \geq 2c{\bf 1}_{[s]} = 2 \max \{c^{\top}u \mid u \in {\rm MP}(L)\}$, as required. 
\end{document}